\documentclass[a4paper]{amsart}
\usepackage{amsmath}
\usepackage{cite}
\usepackage{amsthm}
\usepackage{amssymb}
\usepackage[ansinew]{inputenc}
\usepackage{url}

\makeatletter
\def\paragraph{\@startsection{paragraph}{4}%
  \z@\z@{-\fontdimen2\font}%
  {\normalfont\bfseries}}
\makeatother

\theoremstyle{plain}
\newtheorem{lemma}{Lemma}[section]

\newtheorem{theorem}[lemma]{Theorem}
\newtheorem{cor}[lemma]{Corollary}
\newtheorem{prop}[lemma]{Proposition}

\theoremstyle{definition}
\newtheorem{defi}[lemma]{Definition}
\newtheorem{example}[lemma]{Example}
\newtheorem{prob}[lemma]{Problem}

\newcommand\ov{\overline}

\newcommand\con{\sim_c}
\newcommand\cp{\sim_p}
\newcommand\co{\sim_{\!o}}

\newcommand\pp{\mathbb{P}}

\newcommand\tr{\mbox{\tiny{$R$}}}
\newcommand\tm{\mbox{\tiny{$M$}}}
\newcommand\tit{\mbox{\tiny{$T$}}}
\newcommand\rrh{\rho_{\!\tr}}

\newcommand\tla{\rho({a})}
\newcommand\tlb{\rho({b})}
\newcommand\tlc{\rho({c})}
\newcommand\ra{\to}
\newcommand\sra{\stackrel{*}{\ra}}
\newcommand\wti{\widetilde}

\date{}
\begin{document}

\title{Decidability and Independence of Conjugacy Problems\\ in Finitely Presented Monoids}

\author{Jo\~ao Ara\'ujo}
\address{Universidade Aberta, R. Escola Polit\'{e}cnica, 147,  
1269-001 Lisboa, Portugal}
\address{CEMAT-Ci\^{e}ncias Universidade de Lisboa,  1749-016 Lisboa, Portugal}
\email{jjaraujo@fc.ul.pt}

\author{Michael Kinyon}
\address{Department of Mathematics, University of Denver, Denver, CO 80208}
\email{mkinyon@du.edu}

\author{Janusz Konieczny}
\address{Department of Mathematics, University of Mary Washington,  
Fredericksburg, VA 22401.}
\email{jkoniecz@umw.edu}

\author{Ant\'onio Malheiro}
\address{%
Centro de Matem\'{a}tica e Aplica\c{c}\~{o}es\\
Faculdade de Ci\^{e}ncias e Tecnologia\\
Universidade Nova de Lisboa\\
2829--516 Caparica\\
Portugal
}
\address{%
Departamento de Matem\'{a}tica\\
Faculdade de Ci\^{e}ncias e Tecnologia\\
Universidade Nova de Lisboa\\
2829--516 Caparica\\
Portugal
}
\email{%
ajm@fct.unl.pt
}

%
%

\maketitle

\begin{abstract}
There have been several attempts to extend the notion of conjugacy from groups to monoids.
The aim of this paper is study the decidability and independence of conjugacy problems
for three of these notions (which we will denote by $\cp$, $\co$, and $\con$) in
certain classes of finitely presented monoids. We will show that in the class of polycyclic monoids,
$p$-conjugacy is ``almost'' transitive, $\con$ is strictly included in $\cp$, and
the $p$- and $c$-conjugacy problems are decidable with linear compexity.
For other classes of monoids, the situation is more complicated.
We show that there exists a monoid $M$ defined by a finite complete
presentation such that the $c$-conjugacy problem for $M$ is undecidable, and
that for finitely presented monoids, the $c$-conjugacy problem and the word
problem are independent, as are the  $c$-conjugacy and $p$-conjugacy problems.

\bigskip


\bigskip

\end{abstract}

\section{Introduction}\label{scon}
\setcounter{equation}{0}

The well-known notion of conjugacy from group theory can be extended
to monoids in many different ways. The authors dealt with four notions of conjugacy in monoids
in \cite{AKKM2017,AKM14}. The present paper can be considered an extension of this work. Any generalization of
the conjugacy relation to general monoids must avoid inverses. One of the possible
formulations, spread  by Lallement \cite{La79} for a free monoid $M$, was the following relation:
\begin{equation}\label{econ2}
a\cp b\Leftrightarrow \exists_{u,v\in M}\ a=uv \mbox{ and } b=vu.
\end{equation}
(Lallement credited the idea of the relation $\cp$ to Lyndon and Sch\"{u}tzenberger \cite{lyndon}.)
If $M$ is a free monoid, then $\cp$ is an equivalence relation on $M$
\cite[Corollary~5.2]{La79},
and so it can be regarded as a conjugacy in $M$. In a general monoid $M$, the relation $\cp$ is reflexive and
symmetric, but not transitive.
The transitive  closure $\cp^*$ of $\cp$ has
been defined as a conjugacy relation in a general semigroup \cite{Hi06,KuMa07,KuMa09}.
(If $a\cp b$ in a general monoid, we say that $a$ and $b$ are \emph{primarily
conjugate} \cite{KuMa09}, hence our subscript in $\cp$).

Another relation that can serve as a conjugacy in any monoid is defined as
follows:
\begin{equation}\label{econ3}
a\co b\Leftrightarrow \exists_{g,h\in M}\ ag=gb \mbox{ and } bh=ha.
\end{equation}
This relation was defined by Otto for monoids presented by finite Thue systems
\cite{Ot84},
but it is an equivalence relation in any monoid. Its drawback -- as a
candidate for a conjugacy for general monoids --
is that it reduces to the universal relation $M\times M$ for any monoid $M$
that has a zero.

To remedy the latter problem,  three authors of the present paper
introduced a new notion of conjugacy \cite{AKM14}, which
retains Otto's concept for monoids without zero, but does not reduce to
$M\times M$ if $M$ has a zero.
The main idea was to restrict the set from which conjugators can be chosen.
For a monoid $M$ with zero and $a\in M\setminus \{0\}$, let
$\pp(a)$ be the set $\{ g\in M : (\forall m\in M)\ mag=0\ \Rightarrow\ ma=0
\}$, and define
$\pp(0)$ to be $\{0\}$. If $M$ has no zero,
we agree that $\pp(a)=M$, for every $a\in M$.
Following \cite{AKM14}, we define a relation $\con$ on any monoid $M$ by
\begin{equation}\label{e1dcon}
a\con b\Leftrightarrow \exists_{g\in\pp(a)}\exists_{h\in\pp(b)}\
ag=gb\textnormal{ and }bh=ha.
\end{equation}
The relation $\con$ is an equivalence relation on an arbitrary monoid $M$.
Moreover, if
$M$ is a monoid without zero, then $\con\,\,=\,\,\co$; and if $M$ is a free
monoid, then
$\con\,\,=\,\,\co\,\,=\,\,\cp$. In the case when $M$ has a zero, the conjugacy
class of $0$
with respect to $\con$ is $\{0\}$. Throughout the paper we shall refer to $\sim_i$, where $i\in\{p,o,c\}$,
as $i$-conjugacy.

The aim of this paper is to study the decidability and independence of the $i$-conjugacy problems
in some classes of finitely presented monoids.

It is well-known that the conjugacy problem for finitely presented groups is undecidable; that is,
there exists a finitely presented group for which the conjugacy problem is
undecidable \cite{No54}.
The relations $\cp$, $\co$, and $\con$
reduce to group conjugacy when a monoid is a group.
It follows that the $i$-conjugacy problem, for $i\in\{p,o,c\}$, is also undecidable.
However, it is of interest to study decidability of the $i$-conjugacy problems
in particular classes of finitely presented monoids.

First, we consider the class of polycyclic monoids,
which are are finitely presented monoids with zero.
The polycyclic monoids $P_n$, with $n\geq 2$, were first introduced
by Nivat and Perrot \cite{NP70}, and later rediscovered by Cuntz in the
context of the theory of $C^*$-algebras \cite[Section~1]{Cun77}. (Within the
theory of $C^*$-algebras, the polycyclic monoids are often referred to as Cuntz
inverse semigroups.) The polycyclic monoids appear to be related to the idea of
self-similarity \cite{Hines1998}. For example, the polycyclic monoid $P_2$ can
be represented by partial injective maps on the Cantor set: its two generators,
$p_1$ and $p_2$, map, respectively, the left and right hand sides of the Cantor
set, to the whole Cantor set.
These monoids can also be characterized as  the syntactic monoid of the restricted Dyck
language on a set of cardinality $n$, that is, the language that consists of all
correct bracket sequences of $n$ types of brackets.
The study of representations of the polycyclic monoids naturally connects with
the study of its conjugacy relations \cite{Lawson2009,Jones2012}. In
\cite{Lawson2009}, the classification of the `proper closed inverse submonoids'
of $P_n$ depends on the study of its conjugacy classes.

In Section~\ref{spol}, we characterize $p$-conjugacy and $c$-conjugacy in the polycyclic monoids,
and conclude that $\con\ \subset\ \cp$.  (For sets $A$ and $B$, we write $A\subset B$ if $A$ is a proper subset of $B$.)
We then show that the $p$-conjugacy and $c$-conjugacy problems are decidable for polycyclic monoids,
and that, given words $a$ and $b$, testing whether or not $a\sim_i b$, for $i\in\{p,c\}$, can be
done linearly on the lengths of $a$ and $b$. Note that in a polycyclic monoid $P_n$, the relation $\co$
is universal since $P_n$ has a zero.

These positive results obtained for  polycyclic monoids concerning
the decidability and complexity of the conjugacy problems cannot be extended to
the general finitely presented monoids.

In Section~\ref{sdec}, we study decidability results.
In particular, we show that there exists a monoid $M$ defined by a finite complete
presentation such that the $c$-conjugacy problem for $M$ is undecidable (Proposition~\ref{pzex}).

In Section~\ref{sind}, we study independence results.
The word problem for groups is undecidable \cite{Ma47, No58, Po47}.
However, for groups, the word problem is reducible to the conjugacy problem \cite[page~225]{Ot84},
hence if the conjugacy problem for a group $G$ is decidable, then the word problem
for $G$ is also decidable.
Therefore, the word problem and the conjugacy problem for groups are not
independent.
The situation for monoids is different. Osipova \cite{Os73} has proved that for
finitely presented monoids,
the word problem, the $p$-conjugacy problem, and the $o$-conjugacy problem are
pairwise independent.
We show that for finitely presented monoids, the word problem and the
$c$-conjugacy problem are independent (Theorem~\ref{tin1}),
and that the $p$-conjugacy problem and the $c$-conjugacy problem are also
independent (Theorem~\ref{tin2}).
We do not know if the $c$-conjugacy problem and the $o$-conjugacy problem are
independent.

We conclude the paper with Section~\ref{spro} that lists open problems regarding the conjugacies under
discussion.

\section{Background}
In this section we will formulate the main concepts needed in the following
sections. For further background on the free monoid, see \cite{Ho95}; for
presentations, see \cite{Hi92,Rus95}; and, for rewriting systems, see
\cite{BoOt93}.

\paragraph{Alphabets and words.} Let $\Sigma$ be a
non-empty set, called an \emph{alphabet}. We denote by $\Sigma^*$ the set of
finite strings (called \emph{words}) of elements
of $\Sigma$, including the \emph{empty word} $1$. For $w\in\Sigma^*$ and
$a\in\Sigma$,
we denote by $|w|$ the \emph{length} of the word $w$
and by $|w|_a$ the number of occurrences of $a$ in $w$. For example, if
$\Sigma=\{a,b,c\}$ and $w=aabba\in\Sigma^*$,
then $|w|=5$, $|w|_a=3$, and $|w|_c=0$.

A non-empty word $z$ is said to be a \emph{factor} of $w\in \Sigma^*$ if
$w=uzv$, for some words $u,v\in \Sigma^*$. If $w$, $u$, and $v$ are words with
$w=uv$, then $u$ is called a \emph{prefix} of $w$ and $v$ a \emph{suffix} of $w$; the
word $u$ is said to be a \emph{proper} prefix of $w$ if $v$ is non empty
(the notion of proper suffix is dual). Two
words $u$ and $v$ are said to be \emph{prefix-comparable} if either $u$ is a
prefix of
$v$ or $v$ is a prefix of $u$.

\paragraph{Rewriting systems.} Any subset $R$ of
$\Sigma^*\times \Sigma^*$ is called
a \emph{rewriting system} (or a \emph{Thue system}) on $\Sigma$.
An element $(x,y)$ of $R$, also commonly denoted $x=y$, is called a
\emph{rewriting rule}.
If $(x,y)\in R$ and $u,v\in\Sigma^*$,
we say that $uxv$ \emph{reduces} to $uyv$ and we write
$uxv\ra uyv$. A word $w$ is said to be
\emph{irreducible} if there is no $w'\in \Sigma^*$, such that
$w\ra w'$.
We denote by $\sra $ the reflexive and transitive closure of
$\ra $.

 A rewriting system $R$
on $\Sigma$ is \emph{special} if  every element of $R$
is
of the form $(x,1)$ with $x\ne1$;
it is \emph{monadic} if every element of $R$ is of the form $(x,y)$ with
$y\in\Sigma\cup\{1\}$ and $|x|>|y|$;
it is \emph{length reducing} if $|x|>|y|$ for all $(x,y)\in R$;
it is \emph{noetherian} if there is no infinite sequence $w_1,w_2,w_3,\ldots$
of
words in $\Sigma^*$
such that $w_1\ra w_2\ra w_3\ra \cdots$;
it is \emph{confluent}
if for all $u,v,w\in \Sigma^*$, if $u \sra  v$ and  $u
\sra  w$,
then there exists $z\in \Sigma^*$ such that $v\sra  z$ and
$w\sra  z$; and $R$ is
\emph{complete} if it is both noetherian and confluent. Note that if $R$ is
special or monadic, then it is length reducing,
and if $R$ is length reducing, then it is noetherian.

\paragraph{Monoid presentations.} Every rewriting system $R$ on $\Sigma$
defines a monoid. The set $\Sigma^*$
with
concatenation of words as multiplication
is a monoid, called the \emph{free monoid} on $\Sigma$. Denote by
$\rrh$ the smallest congruence on $\Sigma^*$ containing $R$
(called the \emph{Thue congruence}).
We denote by $M(\Sigma;R)$ the quotient monoid
$\Sigma^*\!/\rrh$.
The elements of $M(\Sigma;R)$ are the congruence classes
$[u]_{\tm}=\{w:w\,\,\rrh\,\,u\}$, where $u\in\Sigma^*$. Whenever
possible and when it is clear from the context, we shall
 omit the brackets to denote congruence classes, and thus identify words with
the elements of the monoid that they represent.

Suppose $M$ is any monoid such that $M\cong M(\Sigma;R)$ (that is, $M$ is
isomorphic to $M(\Sigma;R)$). Then the
pair $(\Sigma; R)$ is a \emph{presentation}
of $M$ with \emph{generators} $\Sigma$ and \emph{defining relations} $R$, and
we say that $M$
is defined by
$(\Sigma; R)$ or simply by $R$. A presentation $(\Sigma; R)$
is said to be \emph{finite} if both $\Sigma$ and $R$ are finite. A monoid $M$
defined by
a finite presentation is called \emph{finitely presented}.

\begin{defi}\label{ddec}
Let $M=M(\Sigma;R)$ be a finitely presented monoid, and let
$\sim_i$ be one of the conjugacy relations under consideration
($i\in\{p,o,c\}$).
We say that the $i$-conjugacy problem for $M$ is \emph{decidable}
if there is an algorithm that given any pair $(u,v)$ of words in $\Sigma^*$,
returns YES if $[u]_{\tm}\sim_i[v]_{\tm}$ and NO otherwise.
If such an algorithm does not exists, we say that the $i$-conjugacy problem for
$M$ is \emph{undecidable}. We have an analogous definition of the decidability
of the word problem for $M$, in
which case we check if $[u]_{\tm}=[v]_{\tm}$.
\end{defi}

\paragraph{Monoids with zero: rewriting systems and presentations.}
Consider a rewriting system $R$ defined on a set $\Sigma_0=\Sigma\cup\{0\}$, where $0$
is a symbol not in $\Sigma$, and a set $R_0$ of rewriting
rules of the form $(x0,0)$, $(0x,0)$ and $(00,0)$, for any $x\in\Sigma$. The
monoid $T=M(\Sigma_0; R\cup R_0)$ is a monoid with zero $[0]_{\tit}$. For
simplicity, we refer to the pair $(\Sigma_0;R)$ as a \emph{monoid-with-zero
presentation} of $T$. Notice that the monoid presentation $(\Sigma_0; R\cup R_0)$ is finite or monadic when $(\Sigma_0;R)$ is finite
or monadic, respectively.

If a monoid $M$ is defined by a presentation $(\Sigma; R)$ then the
monoid $M^0$,  obtained from $M$ by adding a zero element, is defined by the
monoid-with-zero presentation  $(\Sigma; R)$. Observe that
$[0]_{M^0}$ is the zero in $M^0$ and that
$M^0\ne\{[0]_{\tm^0}\}$.
Regarding these presentations, we can deduce
by \cite[Proposition~3.1]{AM2011} that if the rewriting system $ R$ on $\Sigma$
is complete, then so is the new rewriting system $R\cup R_0$ on $\Sigma_0$.

Throughout the text we refer to a presentation as noetherian, confluent,
complete, monadic, etc., whenever the associated rewriting system has the
respective property.

\section{Conjugacy in the polycyclic monoids}\label{spol}
In this section, we study $p$-conjugacy and $c$-conjugacy in the
class of polycyclic monoids, an important class of inverse monoids.  A
monoid $M$ is called
an \emph{inverse monoid} if for every $a\in M$, there exists a unique
$a^{-1}\in M$ (an \emph{inverse} of $a$) such that
$aa^{-1}a=a$ and $a^{-1}aa^{-1}=a^{-1}$ \cite[p.~145]{Ho95}.

In general, $p$-conjugacy is not transitive in inverse semigroups. For
instance, by \cite[Proposition~4.2]{Ch93},
$p$-conjugacy is not transitive in free inverse monoids.
We will show that in the polycyclic monoids, $p$-conjugacy is transitive for
the elements not $\cp$-related to zero,
and that $\con\,\,\subset\,\,\cp$.

We note that in the polycyclic monoids, $\co$ is the universal relation
since every polycyclic monoid has a zero.

\subsection{General properties of the polycyclic monoids}
Let $n\geq 2$. Consider a set $A_n=\{p_1,\ldots, p_n\}$ and denote by
$A_n^{-1}$ a disjoint copy $\{p_1^{-1},\ldots, p_n^{-1}\}$.
Let $\Sigma=A_n\cup A_n^{-1}$. Given any
$x\in \Sigma$, we define $x^{-1}$ to be $p_i^{-1}$ if $x=p_i\in A_n$,
and to be $p_i$ if $x=p_i^{-1}\in A_n^{-1}$. This notation can be extended
to $\Sigma^*$ by setting $(xw)^{-1}=w^{-1}x^{-1}$, for every $x\in \Sigma$ and
$w\in \Sigma^*$, and $1^{-1}=1$.

Denote by $R$ the set of rewriting rules on $\Sigma_0=\Sigma\cup\{0\}$ of the form
$p_i^{-1}p_i=1$, for $i\in\{1,\ldots,n\}$, and of the form $p_i^{-1}p_j=0$, for
$i,j\in\{1,\ldots,n\}$ and $i\neq j$.
Consider the monoid $P_n$ defined by the
monoid-with-zero presentation $(\Sigma_0; R)$.
The monoid $P_n$ is called the \emph{polycyclic monoid} on $n$ generators.
Notice that the given presentation of $P_n$ is monadic, and thus length
reducing.

An irreducible element (with respect to $R$) cannot have a factor of the form
$p_i^{-1}p_j$, for any $i,j\in\{1,\ldots,n\}$. Thus, irreducible
elements have the form $yx^{-1}$, where $y,x\in A_n^*$, or the form $0$.
It is well known (e.g., \cite[subsection~9.3]{Lawson2009}) that every nonzero
element $w$ of $P_n$ has a
\emph{unique} irreducible representation $\ov{w}$ of the form $yx^{-1}$ with $y,x\in A_n^*$.
Therefore, irreducible elements are in one-to-one correspondence with elements
of the polycyclic monoid. We deduce the following:

\begin{lemma}
The monoid-with-zero presentation $(\Sigma_0;R)$ of the polycyclic monoid
$P_n$ is finite and complete.
\end{lemma}

Whenever we write $a=yx^{-1}$, it will be understood that $x,y\in A_n^*$.
Hereafter, we shall identify irreducible elements with the elements of the
polycyclic monoid that they represent.

We will frequently use the following lemma,
which follows from the unique representation of the nonzero elements of $P_n$.
\begin{lemma}
\label{lbas}
Consider  nonzero elements $yx^{-1}$ and $vu^{-1}$ of $P_n$. Then:
\begin{itemize}
  \item[\rm(1)] $yx^{-1}\cdot vu^{-1}\ne0$ iff $x$ and $v$ are
prefix-comparable;
  \item[\rm(2)] if $yx^{-1}\cdot vu^{-1}\ne0$, then
\[
y x^{-1} \cdot v u^{-1} =
\begin{cases}
 \ yzu^{-1} & \text{if } v=xz\,, \\
 \ y(uz)^{-1} & \text{if } x=vz\,.
\end{cases}
\]
  \item[\rm(3)] $y=v$ in $P_n$ iff $y=v$ in $A_n^*$, and $x^{-1}=u^{-1}$ in
$P_n$ iff $x=u$ in $A_n^*$.
\end{itemize}
\end{lemma}

An irreducible word is said to be \emph{cyclically reduced} if it is empty or
zero, or if its first letter $c$ and its last letter $d$ satisfy $c\neq
d^{-1}$. Every nonzero irreducible word can be written in the form
$ryx^{-1}r^{-1}$,  with $r\in A_n^*$ and $yx^{-1} $ a cyclically reduced word.
From any irreducible word $a$ we compute a cyclically reduced word $\widetilde{a}$ in
the following way: if $a$ is cyclically reduced, we let $\widetilde{a}$ be equal to $a$; otherwise,
$a=ryx^{-1}r^{-1}$ as above, so we let $\widetilde{a}$ be the (possibly empty) cyclically reduced word $yx^{-1}$.
We obtain the following fact for any nonzero irreducible word $a\in P_n$:
\begin{equation}\label{charc_reduced}
 a=r\widetilde{a}r^{-1}\ \textrm{for some word } r\in A_n^*\,.
\end{equation}

For each nonzero element $a=yx^{-1}\in P_n$, denote by $\tla$ the irreducible word obtained from $x^{-1}y$.
We also set $\rho(0)=0$. Let $a=yx^{-1}\in P_n$. We record the following facts about $\widetilde{a}$ and $\rho(a)$:
\begin{itemize}
  \item[(a)] $\rho(a)$ is $x^{-1}y$ reduced in $P_n$;
  \item[(b)] $\widetilde{a}$ is $x^{-1}y$ reduced in $(\Sigma,R_1)$, where $R_1=\{(p_i^{-1}p_i,1):i\in\{1,\ldots,n\}$;
  \item[(c)] $\tla$ is cyclically reduced;
  \item[(d)] $\tla=\widetilde{a}\in A_n^*$ if $x$ is a prefix of $y$; $\tla=\widetilde{a}\in (A_n^{-1})^*$ if $y$ is a prefix of $x$;
and $\tla=\rho(\widetilde{a})=0$ otherwise.
\end{itemize}
For example, if $a=p_1p_2p_3^{-1}p_1^{-1}$, then $\widetilde{a}=p_2p_3^{-1}$ and $\rho(a)=0$.

The following lemma can be easily deduced.

\begin{lemma}
\label{l51}
For all $a=pq^{-1}\cdot rs^{-1}\in P_n$, the cyclically reduced word $\widetilde{a}$ is given by
\[
\begin{cases}
\ lt & \text{if } r=qt\text{ and }p=sl\,, \\
\ \widetilde{tl^{-1}} & \text{if }r=qt\text{ and }s=pl\,, \\
\ \widetilde{lt^{-1}} & \text{if }q=rt\text{ and }p=sl\,, \\
\ (lt)^{-1} & \text{if }q=rt\text{ and }s=pl\,.
\end{cases}
\]
\end{lemma}

\subsection{$p$-conjugacy in $P_n$}
We first observe that for every $a\in P_n$,  $a\cp\widetilde{a}$ and $a\cp \tla$
(by the definitions of $\widetilde{a}$ and $\rho(a)$.

\begin{lemma}
\label{l51aa}
Let $a\in P_n$. Then $a\cp 0$ if and only if  $\tla=0$.
\end{lemma}
\begin{proof}
Suppose that $a\cp 0$. If $a=0$ then $\rho(a)=0$. Suppose $a\ne0$. Then
 $0\ne a=pq^{-1}\cdot rs^{-1}$ and $0=rs^{-1}\cdot pq^{-1}$, for some $p,q,r,s\in
A_n^*$. The latter equality implies that
 $p$ and $s$ are not prefix-comparable (by Lemma~\ref{lbas}). And the former
implies that $r=qt$ or $q=rt$, for some $t\in A_n^*$.
Hence, $a=pts^{-1}$ (if $r=qt$) or $a=p(st)^{-1}$ (if $q=rt$).
Suppose that $a=pts^{-1}$. If $pt$ is a prefix of $s$, then also $p$ is a prefix
of $s$; and
if $s$ is a prefix of $pt$, then  either $s$ is a prefix of $p$ or $p$ is a
prefix of $s$. It follows that
neither $pt$ is a prefix of $s$ nor $s$ is a prefix of $pt$, which implies
$\tla=0$. By a similar argument,
we obtain $\tla=0$ if $a=p(st)^{-1}$.
The converse follows from the fact that $a\cp\tla$.
\end{proof}

\begin{lemma}
\label{l453}
Let $a$ and $b$ be nonzero elements of $P_n$ with $\tla=\tlb=0$.
Then $a\cp b$ if and only if $\widetilde{a}=\widetilde{b}$.
\end{lemma}
\begin{proof}
Suppose that  $a\cp b$. Then there exist elements $x,y,u,v\in A_n^*$ such that
$a=yx^{-1}\cdot vu^{-1}$ and $b=vu^{-1}\cdot yx^{-1}$.
Since $a\ne0$, we know that $x$ and $v$ are prefix-comparable. Similarly, since $b\ne0$, $u$ and $y$ are prefix-comparable.

Suppose first that $x$ is a prefix of $v$ and $u$ is a prefix of $y$, that is,
$v=xp$ and $y=uq$ for some $p,q\in A_n^*$. Hence $a=uqx^{-1}xpu^{-1}=uqpu^{-1}$,
and so $\tla =qp\neq 0$, which is a contradiction.
Similarly, we obtain a contradiction if we assume that $v$ is a prefix of $x$
and $y$ is a prefix of $u$.

Suppose that $x$ is a prefix of $v$ and that $y$ is a prefix of $u$, that is,
$v=xp$ and $u=yq$ for some $p,q\in A_n^*$. Then
$a=yx^{-1}xpq^{-1}y^{-1}=y(pq^{-1})y^{-1}$ and
$b=xpq^{-1}y^{-1}yx^{-1}=x(pq^{-1})x^{-1}$. Thus $\widetilde{a}=\widetilde{pq^{-1}}=\widetilde{b}$ as
required.
In a similar way we obtain the intended result if $v$ is a prefix of $x$ and $u$
is a prefix of $y$.

The converse follows easily by first noticing that $a=r\widetilde{a}r^{-1}$ and
$b=s\widetilde{b}s^{-1}$, for some $r,s\in A_n^*$ by (\ref{charc_reduced}). If
$\widetilde{a}=\widetilde{b}$ we get the required result since $a=r\widetilde{a}s^{-1}\cdot sr^{-1}$
and $b= sr^{-1}\cdot r\widetilde{b}s^{-1}$.
\end{proof}

The following theorem characterizes $p$-conjugacy in $P_n$.

\begin{theorem}\label{ccp}
Let $a,b\in P_n$. Then $a\cp b$ if and only if one of the following conditions
is satisfied:
\begin{itemize}
\item[\textrm(a)] $a=\tlb=0$ or $\tla=b=0$;
\item[\textrm(b)] $\tla=\tlb=0$ and $\widetilde{a}=\widetilde{b}$;
\item[\textrm(c)] $\widetilde{a}, \widetilde{b}\in A_n^*$ and
$\widetilde{a}\cp\widetilde{b}$ in the free monoid
$A_n^*$; or
\item[\textrm(d)] $\widetilde{a}, \widetilde{b}\in (A_n^{-1})^*$ and
$\widetilde{a}\cp\widetilde{b}$ in the free
monoid $(A_n^{-1})^*$.
\end{itemize}
\end{theorem}
\begin{proof}
Suppose that $a\cp b$. If $a=0$ or $b=0$, then (a) holds by Lemma~\ref{l51aa}.
Now assume that $a$ and $b$ are nonzero elements. Then, there are
$p,q,r,s\in A_n^*$ such that $a=pq^{-1}\cdot rs^{-1}$ and $b=rs^{-1}\cdot
pq^{-1}$.
By Lemma~\ref{l51}, we have:
if $r=qt$ and $p=sl$, then $\widetilde{a}=lt$ and $\widetilde{b}=tl$, so (c) holds;
if $q=rt$ and $s=pl$, then $\widetilde{a}=(lt)^{-1}$ and $\widetilde{b}=(tl)^{-1}$, so (d) holds;
if $r=qt$ and $s=pl$, then $\widetilde{a}=\widetilde{b}=\widetilde{tl^{-1}}$; and
if $q=rt$ and $p=sl$, then $\widetilde{a}=\widetilde{b}=\widetilde{lt^{-1}}$.
In the last two cases, we have $\widetilde{a}=\widetilde{b}$, and so $\tla=\tlb$. Thus, in those cases,
either (c) or (d) holds (if $\tla=\tlb\ne0$),
or (b) holds by Lemma~\ref{l453} (if $\tla=\tlb=0$).

Conversely, if (a) holds then $a\cp b$ by Lemma~\ref{l51aa};
and if (b) holds then $a\cp b$ by Lemma~\ref{l453}.
Suppose that (c) holds and let $\widetilde{a}=uv$ and $\widetilde{b}=vu$, where
$u,v\in A_n^*$.
Then $a=puvp^{-1}$ and $b=qvuq^{-1}$ for some $p,q\in A_n^*$ by
(\ref{charc_reduced}).
Hence $a= puq^{-1}\cdot qvp^{-1}$ and $b= qvp^{-1}\cdot puq^{-1}$, and so $a\cp
b$ as required.
The proof in the case when (d) holds is similar.
\end{proof}

It is worth noting that the nonzero idempotents of $P_n$ form a single
$p$-conjugacy class. Indeed, the nonzero idempotents have the form $xx^{-1}$, and
$\widetilde{a}=1$ if and only if $a$ is an idempotent. So, they form a single $p$-conjugacy class
by Theorem~\ref{ccp}.

We recall that $\cp$ is transitive in any free monoid. For the polycyclic
monoid, we have the following result.
\begin{theorem}
\label{p53}
In the polycyclic monoid $P_n$, we have:
\begin{itemize}
\item [\rm(1)] for all $a,b,c\in P_n$ with $b\ne0$, if $a\cp b$ and $b \cp c$,
then $a \cp c$;
\item [\rm(2)] $\cp\,\circ\cp\,\,=\,\,\cp^*$.
\end{itemize}
\end{theorem}
\begin{proof}
\noindent To prove (a), let $a,b,c\in P_n$ with $b\neq 0$.
Suppose that $a\cp b \cp c$.
If $\tlb\neq 0$ then, by Theorem~\ref{ccp}, either $\widetilde{a},
\widetilde{b}, \widetilde{c} \in
A_n^*$ or $\widetilde{a},
\widetilde{b}, \widetilde{c}\in (A_n^{-1})^*$,
and $a\cp c$ follows since $\cp$ is transitive in the free monoid. Suppose that
$\tlb=0$. Then, by Theorem~\ref{ccp}, $\tla=\tlc=0$. Thus,
if $a=0$ or $c=0$, then $a \cp c$. Suppose that $a\neq 0$ and $c\neq 0$. Then,
by Theorem~\ref{ccp}, $\widetilde{a}=\widetilde{b}=\widetilde{c}$, and so, again by Theorem~\ref{ccp},
$a\cp c$.

Statement (2) follows from (1).
\end{proof}

The relations $\cp$ and $\cp^*$ are not equal in $P_n$. For example, consider
the polycyclic monoid $P_2$ with $A_2=\{x,y\}$.
Then, for $a=xx^{-1}$ and $c=yy^{-1}$ in $P_2$, $a\cp 0\cp c$, so $a\cp^{*}c$,
but $(a,c)\notin\,\,\cp$
by Theorem~\ref{ccp}.

\subsection{$c$-conjugacy in $P_n$}
Referring to the definition of $\con$, we begin with a description of the set from
which the conjugators must be chosen.

\begin{lemma}
\label{lpp}
For all $x,y\in A_n^*$, $\pp(yx^{-1})=\{rs^{-1} : r\text{ is a prefix of }x\}$.
\end{lemma}
\begin{proof}
Let $rs^{-1}\in\pp(yx^{-1})$. Then $yx^{-1}\cdot rs^{-1}\ne0$, and so $r$
and  $x$ are prefix-comparable. Suppose that $x$ is a proper prefix of
$r$, that is, $r=xp_it$ for some
$p_i\in A_n=\{p_1,\ldots,p_n\}$ and $t\in A_n^*$. Let $j\in\{1,\ldots,n\}$ with
$j\ne i$.
Then $ (yp_j)^{-1}\cdot yx^{-1}=(xp_j)^{-1}\ne0$, while $(yp_j)^{-1}\cdot
yx^{-1}\cdot rs^{-1}=(xp_j)^{-1}\cdot rs^{-1}=0$
since neither $xp_j$ is a prefix of $r$, nor $r$ is a prefix of $xp_j$. This
contradicts the hypothesis
that $rs^{-1}\in\pp(yx^{-1})$. Therefore, $r$ is a prefix of $x$.

Now, let $rs^{-1}$ be an element of $P_n$ and assume that $r$ is a prefix of
$x$. Then $x=rz$ for some $z\in A_n^*$, which gives $yx^{-1}\cdot
rs^{-1}=y(sz)^{-1}$. Thus, for every $vu^{-1}\in P_n$,
$vu^{-1}\cdot yx^{-1}\cdot rs^{-1}=vu^{-1}\cdot y(sz)^{-1}=0$ iff $vu^{-1}\cdot
yx^{-1}=0$ (see Lemma~\ref{lbas}). Thus $rs^{-1}\in \pp(yx^{-1})$.
\end{proof}

The following theorem characterizes $c$-conjugacy in $P_n$.

\begin{theorem}
\label{p56}
Let $a,b\in P_n$. Then $a\con b$ if and only if one of the following conditions
is satisfied:
\begin{itemize}
\item[\textrm(a)] $a=b=0$;
\item[\textrm(b)] $\widetilde{a}=\widetilde{b}$; or
\item[\textrm(c)] $\widetilde{a}, \widetilde{b}\in (A_n^{-1})^*$ and
$\widetilde{a}\cp\widetilde{b}$ in the free
monoid $(A_n^{-1})^*$.
\end{itemize}
\end{theorem}
\begin{proof}
Suppose that $a\con b$. If $a=0$ or $b=0$, then (a) holds since $[0]_c=\{0\}$.
Suppose $a,b\ne0$. Then, there exist $x,y,u,v\in A_n^*$ such that $a=yx^{-1}$
and $b=vu^{-1}$;
and there exist $r,s\in A_n^*$ such that $rs^{-1}\in\pp(yx^{-1})$ and
$yx^{-1}\cdot rs^{-1}=rs^{-1}\cdot vu^{-1}$.
By Lemma~\ref{lpp}, $x=rz$ for some $z\in A_n^*$. By Lemma~\ref{lbas},
$s$ and $v$ are prefix-comparable.

Suppose $v=sw$ for some $w\in A_n^*$. Then, $yx^{-1}\cdot
rs^{-1}=yz^{-1}r^{-1}rs^{-1}=y(sz)^{-1}$
and $rs^{-1}\cdot vu^{-1}=rs^{-1}swu^{-1}=rwu^{-1}$. Thus, since  $yx^{-1}\cdot
rs^{-1}=rs^{-1}\cdot vu^{-1}$,
we have $y(sz)^{-1}=rwu^{-1}$, and so $y=rw$ and $u=sz$.
Hence,
$a=yx^{-1}=rw(rz)^{-1}=r(wz^{-1})r^{-1}$ and $b=vu^{-1}=sw(sz)^{-1}=
s(wz^{-1})s^{-1}$, and so (b) holds.

Suppose $s=vw$ for some $w\in A_n^*$. Then, $yx^{-1}\cdot rs^{-1}=y(sz)^{-1}$
(as in the previous case)
and $rs^{-1}\cdot vu^{-1}=rw^{-1}v^{-1}vu^{-1}=r(uw)^{-1}$. Thus,
$y(sz)^{-1}=r(uw)^{-1}$, and so $y=r$
and $sz=uw$. Since $s=vw$, we have $vwz=uw$, which implies that $u=vt$ for some
$t\in A_n^*$. Thus, $uw=vtw$,
which implies $vwz=vtw$, and so $wz=tw$.
By \cite[Corollary~5.2]{La79}, $tw=wz$ implies $t\cp z$ in $A_n^*$.
Further,
$\widetilde{a}=\widetilde{yx^{-1}}=\wti{r(rz)^{-1}}=\wti{rz^{-1}r^{-1}}=z^{-1}$
and
$\widetilde{b}=\widetilde{vu^{-1}}=\wti{v(vt)^{-1}}=\wti{vt^{-1}v^{-1}}=t^{-1}$. Hence $\widetilde{a},
\widetilde{b}\in
(A_n^{-1})^*$.
Since $t\cp z$ in the free monoid $A_n^*$,
we have $z^{-1}\cp t^{-1}$ in $(A_n^{-1})^*$, and so
$\widetilde{a}\cp\widetilde{b}$ in
$(A_n^{-1})^*$. Hence
(c) holds.

Conversely, if (a) holds, then clearly $a\con b$. Suppose that (b) holds, that is,
$\widetilde{a}=\widetilde{b}$. Let $r,s\in A_n^*$ be such that
$a=r\widetilde{a}r^{-1}$ and $b=s\widetilde{b}s^{-1}$.
 Then, by Lemma~\ref{lpp}, $rs^{-1}\in\pp(a)$, $sr^{-1}\in\pp(b)$, and
\begin{align*}
a\cdot rs^{-1} &= r\widetilde{a}r^{-1}rs^{-1} = r\widetilde{a}s^{-1}
= rs^{-1}s\widetilde{a}s^{-1} = rs^{-1}\cdot b\,,\\
b\cdot sr^{-1} &= s\widetilde{b}s^{-1}sr^{-1} = s\widetilde{b}r^{-1}
= sr^{-1}r\widetilde{b} r^{-1} = sr^{-1}\cdot a\,.
\end{align*}
Hence $a\con b$.
Now, suppose that (c) holds. Since $\widetilde{a}, \widetilde{b}\in
(A_n^{-1})^*$, then letting $t^{-1}=\widetilde{a}$ and $z^{-1}=\widetilde{b}$
we have
$a=yt^{-1}y^{-1}=y(yt)^{-1}$ and $b=vz^{-1}v^{-1}=v(vz)^{-1}$ for some $y,v\in
A_n^*$. Moreover, we have
$\widetilde{a}\cp\widetilde{b}$ in the free monoid $(A_n^{-1})^*$, and
so $t\cp z$ in
$A_n^*$ as well.
Hence $tw=wz$ and $w't=zw'$ for some $w,w'\in A_n^*$. By Lemma~\ref{lpp},
$y(vw')^{-1}\in\pp(a)$ and $v(yw)^{-1}\in\pp(b)$. Further,
\begin{align*}
a\cdot y(vw')^{-1} &= y(yt)^{-1}y(vw')^{-1} = yt^{-1}(vw')^{-1} = y(vw't)^{-1} = \\
&= y(vzw')^{-1} = y(w')^{-1}(vz)^{-1} = y(vw')^{-1}v(vz)^{-1} = y(vw')^{-1}\cdot b\,,\\
b\cdot v(yw)^{-1} &= v(vz)^{-1}v(yw)^{-1} = vz^{-1}(yw)^{-1} = v(ywz)^{-1} = \\
&= v(ytw)^{-1} = vw^{-1}(yt)^{-1} = v(yw)^{-1}y(yt)^{-1} = v(yw)^{-1}\cdot a\,.
\end{align*}
Hence $a\con b$, which concludes the proof.
\end{proof}

As for $p$-conjugacy, the nonzero idempotents form a single $c$-conjugacy
class (see Theorem~\ref{p56} and the paragraph after Theorem~\ref{ccp}). Moreover, we have the
following strict inclusion between $\con$ and $\cp$.

\begin{cor}
\label{cpc}
In the polycyclic monoid $P_n$, $\con\,\,\subset\,\,\cp$.
\end{cor}
\begin{proof}
The inclusion $\con\,\,\subseteq\,\,\cp$ follows by Theorems~\ref{ccp} and~\ref{p56}.
To show that $\con$ is properly contained in $\cp$, consider two distinct
generators $x$ and $y$ in $A_n$.
Let $a=xxyx^{-1}$ and $b=yyxy^{-1}$ in $P_n$. Then $\widetilde{a}=xy$ and
$\widetilde{b}=yx$.
Hence $\widetilde{a}\cp\widetilde{b}$ in the free monoid $A_n^*$,
and so $a\cp b$ in $P_n$ by Theorem~\ref{ccp}. On the other hand, none of (a), (b), or (c)
of Theorem~\ref{p56} holds for $a$ and $b$, and so $a\not\con b$ in $P_n$.
\end{proof}

\subsection{Decidability and complexity of conjugacy in $P_n$}

It is known that for free monoids, the $p$-conjugacy problem is decidable in linear
time \cite[Theorem~2.5]{AM1980}. We will show that the same result is true for
the $p$-conjugacy and $c$-conjugacy problems for the polycyclic monoids.

The following lemma is a special case of \cite[Theorem~4.1]{Bo1982}.

\begin{lemma}\label{lemma:irreducible_linear}
Let  $(\Sigma_0;R)$ be the monoid-with-zero presentation of $P_n$, and
let $w\in \Sigma_0^*$.  Then the irreducible element $\overline{w}\in \Sigma_0^*$ such that
$w\ra \overline{w}$ in $P_n$ can be computed in time $O(|w|)$.
\end{lemma}

(For more details on  the big-O notation used in Lemma~\ref{lemma:irreducible_linear}, and more
generally for basic notions on complexity theory, see
\cite[Section~7]{Sip96}.)

\begin{lemma}\label{lemma:cyclically_linear}
Let $a$ be an irreducible word of $P_n$. Then the words $\widetilde{a}$ and $\tla$, can be computed in time
$O(|a|)$.
\end{lemma}

\begin{proof}
The result is obvious if $a=0$.
Let $a=yx^{-1}$. To compute $\widetilde{a}$ proceed as follows:
\begin{enumerate}
 \item compute the  word $x^{-1}y$;
 \item reduce $x^{-1}y$ to an irreducible word $u^{-1}v$ in $(\Sigma,R_1)$ (see (b) above);
 \item output the word $\widetilde{a}=vu^{-1}$.
\end{enumerate}
    To compute $\tla$ proceed in the same way to obtain the word $vu^{-1}$, and next proceed as follows:
\begin{enumerate}
 \item[(4)] if $v$ and $u$ are non-empty, then output $\tla=0$,
otherwise output $\tla=\widetilde{a}$.
\end{enumerate}

We show that each stage of this algorithm uses $O(|a|)$ steps, and so
the result holds. For the first stage, it is sufficient to scan through the word
$yx^{-1}$ (from left to right), detect the first symbol in $(A_n^{-1})^*$, and
output the symbols of $x^{-1}$ followed by the symbols of~$y$. This requires
$O(|a|)$ steps. The third stage is similar. For the second stage, since $R_1$ is
length reducing, we conclude
by \cite[Theorem~4.1]{Bo1982} that $\widetilde{a}$ can be computed in
$O(|a|)$ steps.  Checking if a word is empty can be done in constant time, and
so $\tla$ can be computed in linear time as well.
\end{proof}

\begin{theorem}
Let  $(\Sigma_0;R)$ be the monoid-with-zero presentation of $P_n$, and let $i\in\{p,c\}$.
Then, given two words $x,y\in \Sigma_0^*$, it can be tested in time $O(m)$, where
$m=\max\{|x|,|y|\}$, whether or not $x\sim_iy$ holds in $P_n$.
\end{theorem}
\begin{proof}
Let $x,y\in \Sigma_0^*$. By Lemma~\ref{lemma:irreducible_linear}, the
irreducible words $\overline{x}=a$ and $\overline{y}=b$ can be computed in time $O(m)$, where
$m=\max\{|x|,|y|\}$. Note that $|a|\leq |x|$ and $|b|\leq |y|$.
By Lemma~\ref{lemma:cyclically_linear}, each of the words $\wti{a}$, $\wti{b}$, $\rho(a)$, and $\rho(b)$
can be computed in time $O(m)$.

According to Theorems~\ref{ccp} and \ref{p56}, in order to check whether or
not $x\sim_iy$ holds it suffices to compute $a$, $b$, $\widetilde{a}$,
$\widetilde{b}$, $\tla$, and $\tlb$, and check whether or not they are equal
(as words) or $p$-conjugate (in the free monoid). Since the $p$-conjugacy
problem in the free monoids is decidable in linear time, we deduce the desired
result.
\end{proof}

\section{Decidability in finitely presented monoids}\label{sdec}
In this section, we discuss the decidability of $i$-conjugacy problems in some
classes of finitely presented monoids.

\paragraph{Separation of conjugacies.}
Let $M$ be a monoid without zero.
Consider the monoid $M^0$ obtained from $M$ by adjoining a zero element. It is
immediate that $\sim_{\!o}$ is the
universal relation in $M^0$, while $\sim_c$ is not universal in $M^0$. Now,
$M^0$ has no zero divisors, and hence any two given elements $a$ and $b$ of $M$
are $c$-conjugate in $M^0$ if and only if they are $o$-conjugate in $M$.
Therefore, $\sim_c^{M^0}\,\, =\,\, \sim_{\!o}^M\cup\{(0,0)\}$ in $M^0$.
Similarly, for $p$-conjugacy we have $\sim_p^{M^0}\,\,=\,\, \sim_p^M\cup
\{(0,0)\} $.
Thus, if we identify a monoid $M$ for which $\sim_{\!o}\,\, \ne\,\, \sim_p$ in
$M$,
we then immediately obtain an example of a monoid $M^0$ where
$\sim_{\!c}\,\,\ne\,\, \sim_p$, $\sim_{\!o}\,\,\ne\,\, \sim_p$, and
$\sim_{\!o}\,\,\ne\,\, \sim_c$.
To find such a monoid (within a certain class of rewriting systems), consider
the following example from \cite[Example~2.2]{Ot84}.
\begin{example}
Let $M$ be the monoid defined by the monadic and confluent presentation
$(\Sigma; R)$ with $\Sigma=\{a,b,c\}$ and $R=\{ (ab,b),(cb,b)\}$. As explained
in \cite[Example~2.2]{Ot84}, we have  $bac \sim_p ba$, but clearly $bac$ and
$ba $ are not $o$-conjugate. Therefore, $\sim_{\!o} \,\, \ne\,\, \sim_p$ in $M$,
and hence the relations
$\sim_{\!o} $, $\sim_p$ and $\sim_c$ are pairwise distinct in $M^0$.
\end{example}
We deduce that for monoids defined by monadic presentations, the relations
$\sim_c$, $\sim_p$ and $\sim_{\!o}$ may be different, even when such systems
are also finite and confluent.

\paragraph{Finite complete presentations.}
Narendran and Otto \cite[Lemma~3.6]{NaOt86} constructed a finite
complete presentation $(\Sigma;R)$ such that the $o$-conjugacy problem is
undecidable for the monoid $M=M(\Sigma;R)$.
Using the above observation, we obtain the following result.

\begin{prop}\label{pzex}
There is a monoid defined by a finite complete presentation for which the
$c$-conjugacy problem is undecidable.
\end{prop}
\begin{proof}
Consider the monoid $M^0$ obtained from the monoid $M$ defined by Narendran and
Otto in \cite[page~35]{NaOt86} which has undecidable $o$-conjugacy problem.
Since $M$ is defined by a finite complete presentation, the monoid $M^0$ is
also defined by a finite complete presentation by \cite[Proposition
3.1]{AM2011}. It can be seen that $M$ does not have a zero.
Thus $\sim_c^{M^0}\,\, =\,\, \sim_{\!o}^M\cup\{(0,0)\}$, and hence $M^0$ has
undecidable $c$-conjugacy problem.
\end{proof}

\paragraph{Special presentations.}
It is easy to see that a monoid defined by  a special presentation has a
zero if and only if it is trivial.
Hence, within this class we have $\sim_c\,\,=\,\,\sim_{\!o}$. Zhang
\cite[Theorem~3.2]{Zh91} proved that in every monoid $M$ defined by a special
presentation, the relations $\sim_p$ and $\sim_{\!o}$ also coincide.
Otto \cite[Theorem~3.8]{Ot84} proved that if $M$ is a monoid defined by a
finite, special, and confluent
presentation, then the $o$-conjugacy problem for $M$ is decidable (and so
the $p$-conjugacy and $c$-conjugacy
problems are also decidable for $M$).

\paragraph{One-relator monoids.}
A monoid $M$ is called a \emph{one-relator} monoid if it admits a finite
presentation with one defining relation,
which we will write as $(\Sigma;u=v)$ instead of $(\Sigma,\{(u,v)\})$.
Many decision problems have been studied in the class of one-relator monoids.
For example, it is decidable whether a one-relator monoid has a zero
\cite[Proposition~14]{CaMa09}.
Moreover, a one-relator monoid $M$ containing a zero
admits a presentation $(\{a\}; a^{k+1}=a^k)$, where $k$ is a positive integer
\cite[the proof of Proposition~14]{CaMa09}.
It is easy to check that in this monoid $\sim_p\,\,=\,\,\sim_c\,\,=\{(x,x):x\in
M\}$ and $\sim_{\!o}\,\,=M\times M$.

By the foregoing argument, if $M$ is a one-relator monoid with a zero, then the
$c$-conjugacy and $o$-conjugacy problems
for $M$ are decidable. If $M$ has no zero, then $\sim_c\,\,=\,\,\sim_{\!o}$.
Therefore, the $c$-conjugacy problem for such an $M$ is decidable if and only if the
$o$-conjugacy problem for $M$ is decidable.

Some specific results concerning the decidability of the $o$-conjugacy problem
for this class can be found in \cite{Zh91,Zh92}.

\section{Independence in finitely presented monoids}\label{sind}
In this section, we prove that
for finitely presented monoids, the word problem and the $c$-conjugacy problem
are independent, and that the $p$-conjugacy problem and the $c$-conjugacy problem are
independent.

\begin{defi}
\label{dind}
Decision problems $P_1$ and $P_2$ are \emph{independent}
if there exist finitely presented monoids $M_1$ and $M_2$ such that for $M_1$,
$P_1$ is decidable and $P_2$
is undecidable; and for $M_2$, $P_2$ is decidable and $P_1$ is undecidable.
\end{defi}

\begin{theorem}\label{tin1}
For finitely presented monoids, the word problem and the $c$-conjugacy problem
are independent.
\end{theorem}
\begin{proof}
First, there are finitely presented groups with decidable word problem
but undecidable conjugacy problem \cite{Bo68,Co69}. Let $G$ be a finitely
presented group. A finite group
presentation of $G$ can be effectively converted to a finite (special) monoid
presentation $(\Sigma;R)$
such that $G\cong M(\Sigma;R)$. It follows that there is a monoid $M$
defined by a finite presentation for which the word problem is decidable
and the $c$-conjugacy problem is undecidable.

We will construct a finitely presented monoid for which the converse is true.
Let $G=M(\Sigma;R)$ be a finitely presented group with undecidable word problem
(see \cite{No58}), where $(\Sigma;R)$ is a monoid
presentation. Let $a$ and $b$ be symbols not in $\Sigma$, and let $M=M(A;T) $ be the
monoid defined by the presentation $(A;T)$, where
\begin{align*}
A &= \Sigma\cup\{a,b\}\,,\\
T &= R\cup\{(xa,ax):x\in\Sigma\}\cup\{(bx,b):x\in\Sigma\cup\{a\}\}\cup
\{(xb,b):x\in\Sigma\}\cup\{(aa,a)\}\,.
\end{align*}
Notice that $G$ is a subgroup of $M$.
The word problem for $M$ is undecidable (since otherwise it would be decidable
for $G$).
It is easy to see that $M$ has no zero and that each congruence class
$[u]=[u]_M$ has a representative of the form
$b^p$, $aw$, $ab^p$, or $w$, where $p$ is a positive integer and $w\in\Sigma^*$.

Observe that whenever a rewriting rule from $T$ is applied to a word in $A^*$, the number of occurrences of
$b$ does not change. Thus, for all $u_1,u_2\in A^*$,
if $[u_1]=[u_2]$, then $|u_1|_b=|u_2|_b$. Let $[u],[v]\in M$.
Suppose $[u]\sim_c [v]$. Then $[u][t]=[t][v]$ for some $t\in A^*$. Thus
$[ut]=[tv]$, and so $|u|_b=|v|_b$
by the foregoing observation.

Conversely, suppose $|u|_b=|v|_b$. If $|u|_b=|v|_b=0$, then $[u]\sim_c[v]$
since
$[u][ab]=[ab]=[ab][v]$ and
$[v][ab]=[ab]=[ab][u]$. Suppose $|u|_b=|v|_b=p>0$. If $[u]=[v]$, then
$[u]\sim_c[v]$.
Suppose $[u]\ne[v]$. Then $[u]=[b^p]$ and $[v]=[ab^p]$, or vice versa.
We may assume that $[u]=[b^p]$ and $[v]=[ab^p]$. Then $[u]\sim_c[v]$ since
$[u][b]=[b^{p+1}]=[b][v]$ and
$[v][a]=[ab^p]=[a][u]$.

We have proved that for all $u,v\in A^*$, $[u]\sim_c[v]$ if and only if
$|u|_b=|v|_b$. Hence the $c$-conjugacy problem for $M$
is decidable.
\end{proof}

\begin{theorem}\label{tin2}
For finitely presented monoids, the $p$-conjugacy problem and the $c$-conjugacy
problem are independent.
\end{theorem}
\begin{proof}
Let $M=M(A;T)$ be the monoid from the proof of Theorem~\ref{tin1}. For
$w\in\Sigma^*$, we will write $[w]=[w]_M$
for the element of the monoid $M$, and $[w]_G$ for the element of the group $G$.

Let $u,v\in \Sigma^*$.
Suppose $[u]\sim_p[v]$, that is, $[u]=[s][t]$ and $[v]=[t][s]$ for some $s,t\in
A^*$. The words $s$ and $t$ cannot contain $b$
since in the presentation $(A;T)$ a word with $b$ cannot be reduced to a
word without $b$. But then $s$ and $t$ cannot contain $a$ either
since a word with $a$ cannot be reduced to a word without $a$ unless $b$ is
also
present. It follows
that $[u]_G=[s]_G[t]_G$ and $[v]_G=[t]_G[s]_G$, and so $[u]_G\sim_p[v]_G$.

We have proved that for all $u,v\in\Sigma^*$, if $[u]\sim_p[v]$ in $M$, then
$[u]_G\sim_p[v]_G$ in $G$. The converse is clearly true.
Since $\sim_p$ in $G$ is the group conjugacy and $G$ has undecidable word
problem
(and so undecidable conjugacy problem),
it follows that the $p$-conjugacy problem for $M$ is undecidable. We have
already established in the proof of Theorem~\ref{tin1}
that the $c$-conjugacy problem for $M$ is decidable.

We will now present a monoid that has decidable $p$-conjugacy problem and
undecidable $c$-conjugacy problem.
Osipova \cite{Os73} showed that there exists a finitely presented monoid $M$
that has decidable $p$-conjugacy problem and undecidable $l$-conjugacy problem,
where the $l$-conjugacy stands for the following relation $\sim_l$: given
$a,b\in M$,  $a\sim_l b$ if and only if there exists $g\in M$ such that
$ag=gb$.
Osipova's proof follows the following steps (we use the original notation):
(i) she considers a finitely presented monoid $\Pi_1=M({\mathcal U}_1; {\mathcal
B}_0)$
with undecidable $p$-conjugacy problem; (ii)
she extends the alphabet ${\mathcal U}_1$ to $\mathcal{U}_3={\mathcal
U}_1\cup\{c,d,e_1,\ldots, e_m\}$,
where $m=|{\mathcal U}_1|+2|{\mathcal B}_0|$, and builds  a new  finitely
presented
monoid
$\Pi_3=M({\mathcal U}_3; {\mathcal B}_3)$;
(iii) she shows \cite[Lemma 4]{Os73} that for all words $Q,R\in {\mathcal
U}_1^*$,
$Q\sim_p R$ in $\Pi_1$ if and only if there exists $X\in {\mathcal U}_3^*$
such that $XcQd=cRdX$ in $\Pi_3$; (iv) she concludes \cite[Theorem 2]{Os73}
that the $l$-conjugacy problem for $\Pi_3$ is undecidable;
(v) she shows \cite[Theorem 3]{Os73} that the $p$-conjugacy problem for $\Pi_3$
is decidable.

Now, notice that $\sim_p$ is symmetric, and hence, by \cite[Lemma 4]{Os73}, for
all words $Q,R\in {\mathcal U}_1^*$,
we have $Q\sim_p R$ in $\Pi_1$ if and only if there exist $X,Y\in {\mathcal
U}_3^*$
such that $XcQd=cRdX$ and $YcRd =cQdY$ in $\Pi_3$.
Equivalently,
$Q\sim_p R$ in $\Pi_1$ if and only if $cQd \sim_{\!o} cRd$ in $\Pi_3$.
Therefore, $\Pi_3$ has undecidable $o$-conjugacy problem.

The set ${\mathcal B}_3$ of $\Pi_3$ has rewriting rules of the form
$(e_icG_i,ce_i)$, $(e_ib_j,b_je_i)$, and
$(e_id,G'_ide_i)$, where $i=1,\ldots,m$ and $j=1,\ldots,n$, the $b_j$ are the
letters of the alphabet $\mathcal{U}_1$,
and the $G_i$ and $G'_i$ are fixed words in $\mathcal{U}_1^*$ \cite[pages~70
and~71]{Os73}.
From the form of these rules, we can easily deduce that any two words in
$\mathcal{U}_3^*$
that are equal in $\Pi_3$ have
the same number of occurrences of the letter~$c$. Therefore, $\Pi_3$ does not
have a zero
since the zero element, say $[z]$, would satisfy the identity $[z][c]=[z]$,
contradicting the above observation.
Hence $\sim_{\!o}\,=\,\sim_c$, and hence $\Pi_3$ has undecidable $c$-conjugacy
problem.
\end{proof}

We do not know if the $c$-conjugacy problem and the $o$-conjugacy problem are
independent for finitely presented monoids.
Consider a finitely presented monoid $M$ without zero that has undecidable
$c$-conjugacy problem. Let $M^0$
be the monoid $M$ with a zero $0$ adjoined. Then $M^0$ is finitely presented
and
the $c$-conjugacy problem
for $M^0$ is undecidable (since for all $a,b\in M$, $a\sim_c b$ in $M^0$ if and
only $a\sim_c b$ in $M$).
On the other hand, the $o$-conjugacy problem for $M^0$ is decidable since
$\sim_{\!o}\,=M^0\times M^0$.

Now, suppose $M$ is a finitely presented monoid that has decidable $c$-conjugacy problem. Then, if
we could prove that $M$ has a zero, then the algorithm that always says YES would decide if $[u]_M\sim_o[v]_M$ for all $[u]_M,[v]_M\in M$. Further, if we could prove that $M$ has no zero, then the
algorithm that works for $\sim_c$ would also work for $\sim_o$. However, suppose that the statement ``$M$ has a zero" can neither be proved nor disproved. Then it is conceivable that no algorithm for $o$-conjugacy problem in $M$ exists, that is, that $o$-conjugacy problem is undecidable for $M$.

\section{Open problems}\label{spro}

We conclude this paper with some natural questions related to conjugacy and
presentations.
As we have noticed in Section~\ref{sind},
the independence of the $c$-conjugacy and $o$-conjugacy problems
is related to the decidability of a monoid having a zero. Hence whether
the $o$-conjugacy and $c$-conjugacy problems are independent
hinges on the answer to the following question.

\begin{prob}
Does there exist a finitely presented monoid $M$ for which it is undecidable if
it has a zero, the $o$-conjugacy problem for $M$ is undecidable,
and the $c$-conjugacy problem for $M$ is decidable?
\end{prob}

The word problem is decidable for certain restricted classes of finitely
presented monoids, in particular those admitting a finite complete
presentation. It is then natural to consider this property as a useful tool in
proving decidability results.
In the class of monoids defined by finite, length-reducing, and confluent
rewriting systems,
the $o$-conjugacy problem is decidable \cite[Corollary~2.7]{NaOt85}.
It is also decidable if such monoids have a zero. However, the $p$-conjugacy
problem is undecidable in this class \cite[Corollary~2.4]{NaOt86}.

\begin{prob}
Is the $c$-conjugacy problem decidable for the class of monoids defined by
finite, length-reducing, and confluent rewriting systems?
\end{prob}

This problem could be approached by first considering the class of finite
monadic confluent rewriting systems, as it is the case of polyclyclic monoids.

\begin{prob}
Is the $c$-conjugacy [$p$-conjugacy] problem decidable for the class of monoids
defined by finite, monadic, and confluent rewriting systems?
\end{prob}

\section{Acknowledgments}
The first and second authors acknowledge that this work was developed within FCT 
projects CAUL (PEst-OE/MAT/UI0143/2014) and  CEMAT-CI\^{E}NCIAS 
(UID/Multi/04621/2013). 

The second author was also supported by Simons Foundation 
Collaboration Grant 359872.

The fourth author was supported by the 
Funda\c{c}\~{a}o para
a Ci\^{e}ncia e a Tecnologia (Portuguese Foundation for Science and Technology) 
through the project {\sc UID}/{\sc
  MAT}/00297/2013 (Centro de Matem\'{a}tica e Aplica\c{c}\~{o}es).
  
 Finnaly, the first, second and fourth authors were supported by FCT through 
project ``Hilbert's 24th problem'' (PTDC/MHC-FIL/2583/2014).

\end{document}